\documentclass[12pt]{amsart}
\usepackage{amssymb,latexsym,tikz}
\usepackage[dvips]{epsfig}
\usepackage[usenames,dvipsnames]{pstricks}
\usepackage{hyperref}

\begin{document}
\title[Co-letterplace ideals and Bier spheres]{Resolutions of co-letterplace ideals and generalizations
of Bier spheres}

\author{Alessio D'Al{\`i}}
\address{Dipartimento di Matematica\\
         Universit{\`a} degli Studi di Genova\\
         Via Dodecaneso 35\\
         16146 Genova\\
         Italy}
\email{dali@dima.unige.it}

\author{Gunnar Fl{\o}ystad}
\address{Universitetet i Bergen\\ 
         Matematisk institutt \\
        Postboks 7803\\
        5020 Bergen \\
        Norway} 
\email{gunnar@mi.uib.no}

\author{Amin Nematbakhsh}
\address{School of Mathematics\\
         Institute for Research in Fundamental Sciences (IPM)\\
         P.O. Box: 19395-5746\\
         Tehran\\
         Iran}
\email{nematbakhsh@ipm.ir}


\date{\today}

\begin{abstract}

We give the resolutions of co-letterplace ideals of posets
in a completely explicit, very simple form. This 
generalizes and simplifies a number of linear resolutions in the literature,
among them the Eliahou-Kervaire resolutions of strongly stable ideals 
generated in a single degree. Our method is based on a general result
of K.~Yanagawa using the canonical module of a Cohen-Macaulay Stanley-Reisner
ring. We discuss in detail how the canonical module may effectively be
computed, and from this derive directly the resolutions.

A surprising consequence is that we obtain a large class
of simplicial spheres comprehensively generalizing Bier spheres.
\end{abstract}

\maketitle


\theoremstyle{plain}
\newtheorem{theorem}{Theorem}[section]
\newtheorem{corollary}[theorem]{Corollary}
\newtheorem*{main}{Main Theorem}
\newtheorem{lemma}[theorem]{Lemma}
\newtheorem{proposition}[theorem]{Proposition}

\theoremstyle{definition}
\newtheorem{definition}[theorem]{Definition}
\newtheorem{fact}{Fact}

\theoremstyle{remark}
\newtheorem{notation}[theorem]{Notation}
\newtheorem{remark}[theorem]{Remark}
\newtheorem{example}[theorem]{Example}
\newtheorem{claim}{Claim}


\newcommand{\psp}[1]{{{\bf P}^{#1}}}
\newcommand{\psr}[1]{{\bf P}(#1)}
\newcommand{\op}{{\mathcal O}}
\newcommand{\opw}{\op_{\psr{W}}}
\newcommand{\go}{\op}

\newcommand{\ini}[1]{\text{in}(#1)}
\newcommand{\gin}[1]{\text{gin}(#1)}
\newcommand{\kr}{{\Bbbk}}
\newcommand{\pd}{\partial}
\newcommand{\vardel}{\partial}
\renewcommand{\tt}{{\bf t}}


\newcommand{\coh}{{{\text{{\rm coh}}}}}


\newcommand{\modv}[1]{{#1}\text{-{mod}}}
\newcommand{\modstab}[1]{{#1}-\underline{\text{mod}}}

\newcommand{\sut}{{}^{\tau}}
\newcommand{\sumit}{{}^{-\tau}}
\newcommand{\til}{\thicksim}

\newcommand{\totp}{\text{Tot}^{\prod}}
\newcommand{\dsum}{\bigoplus}
\newcommand{\dprod}{\prod}
\newcommand{\lsum}{\oplus}
\newcommand{\lprod}{\Pi}

\newcommand{\La}{{\Lambda}}

\newcommand{\sirstj}{\circledast}

\newcommand{\she}{\EuScript{S}\text{h}}
\newcommand{\cm}{\EuScript{CM}}
\newcommand{\cmd}{\EuScript{CM}^\dagger}
\newcommand{\cmri}{\EuScript{CM}^\circ}
\newcommand{\cler}{\EuScript{CL}}
\newcommand{\clerd}{\EuScript{CL}^\dagger}
\newcommand{\clerri}{\EuScript{CL}^\circ}
\newcommand{\gor}{\EuScript{G}}
\newcommand{\gF}{\mathcal{F}}
\newcommand{\gG}{\mathcal{G}}
\newcommand{\gM}{\mathcal{M}}
\newcommand{\gE}{\mathcal{E}}
\newcommand{\gD}{\mathcal{D}}
\newcommand{\gI}{\mathcal{I}}
\newcommand{\gP}{\mathcal{P}}
\newcommand{\gK}{\mathcal{K}}
\newcommand{\gL}{\mathcal{L}}
\newcommand{\gS}{\mathcal{S}}
\newcommand{\gC}{\mathcal{C}}
\newcommand{\gO}{\mathcal{O}}
\newcommand{\gJ}{\mathcal{J}}
\newcommand{\gU}{\mathcal{U}}
\newcommand{\mm}{\mathfrak{m}}
\newcommand{\cP}{\mathcal P}
\newcommand{\cS}{\mathcal S}
\newcommand{\St}{\mathcal B}
\newcommand{\ovgL}{\overline{\gL}}
\newcommand{\ovS}{\overline{S}}

\newcommand{\dlim} {\varinjlim}
\newcommand{\ilim} {\varprojlim}

\newcommand{\CM}{\text{CM}}
\newcommand{\Mon}{\text{Mon}}


\newcommand{\Kom}{\text{Kom}}


\newcommand{\EH}{{\mathbf H}}
\newcommand{\res}{\text{res}}
\newcommand{\Hom}{\text{Hom}}
\newcommand{\inhom}{{\underline{\text{Hom}}}}
\newcommand{\Ext}{\text{Ext}}
\newcommand{\Tor}{\text{Tor}}
\newcommand{\ghom}{\mathcal{H}om}
\newcommand{\gext}{\mathcal{E}xt}
\newcommand{\id}{\text{{id}}}
\newcommand{\im}{\text{im}\,}
\newcommand{\codim} {\text{codim}\,}
\newcommand{\resol}{\text{resol}\,}
\newcommand{\rank}{\text{rank}\,}
\newcommand{\lpd}{\text{lpd}\,}
\newcommand{\coker}{\text{coker}\,}
\newcommand{\supp}{\text{supp}\,}
\newcommand{\Ad}{A_\cdot}
\newcommand{\Bd}{B_\cdot}
\newcommand{\Fd}{F_\cdot}
\newcommand{\Gd}{G_\cdot}


\newcommand{\sus}{\subseteq}
\newcommand{\sups}{\supseteq}
\newcommand{\pil}{\rightarrow}
\newcommand{\vpil}{\leftarrow}
\newcommand{\rpil}{\leftarrow}
\newcommand{\lpil}{\longrightarrow}
\newcommand{\inpil}{\hookrightarrow}
\newcommand{\pils}{\twoheadrightarrow}
\newcommand{\projpil}{\dashrightarrow}
\newcommand{\dotpil}{\dashrightarrow}
\newcommand{\adj}[2]{\overset{#1}{\underset{#2}{\rightleftarrows}}}
\newcommand{\mto}[1]{\stackrel{#1}\longrightarrow}
\newcommand{\vmto}[1]{\stackrel{#1}\longleftarrow}
\newcommand{\mtoelm}[1]{\stackrel{#1}\mapsto}

\newcommand{\eqv}{\Leftrightarrow}
\newcommand{\impl}{\Rightarrow}

\newcommand{\iso}{\cong}
\newcommand{\te}{\otimes}
\newcommand{\into}[1]{\hookrightarrow{#1}}
\newcommand{\ekv}{\Leftrightarrow}
\newcommand{\equi}{\simeq}
\newcommand{\isopil}{\overset{\cong}{\lpil}}
\newcommand{\equipil}{\overset{\equi}{\lpil}}
\newcommand{\ispil}{\isopil}
\newcommand{\vvi}{\langle}
\newcommand{\hvi}{\rangle}
\newcommand{\susneq}{\subsetneq}
\newcommand{\sgn}{\text{sign}}
\newcommand{\bihom}[2]{\overset{#1}{\underset{#2}{\rightleftarrows}}}


\newcommand{\xd}{\check{x}}
\newcommand{\ortog}{\bot}
\newcommand{\tL}{\tilde{L}}
\newcommand{\tM}{\tilde{M}}
\newcommand{\tH}{\tilde{H}}
\newcommand{\tvH}{\widetilde{H}}
\newcommand{\tvh}{\widetilde{h}}
\newcommand{\tV}{\tilde{V}}
\newcommand{\tS}{\tilde{S}}
\newcommand{\tT}{\tilde{T}}
\newcommand{\tR}{\tilde{R}}
\newcommand{\tf}{\tilde{f}}
\newcommand{\ts}{\tilde{s}}
\newcommand{\tp}{\tilde{p}}
\newcommand{\tr}{\tilde{r}}
\newcommand{\tfst}{\tilde{f}_*}
\newcommand{\empt}{\emptyset}
\newcommand{\bfa}{{\bf a}}
\newcommand{\bfb}{{\bf b}}
\newcommand{\bfd}{{\bf d}}
\newcommand{\bfl}{{\bf \ell}}
\newcommand{\la}{\lambda}
\newcommand{\bfen}{{\mathbf 1}}
\newcommand{\ep}{\epsilon}
\newcommand{\en}{r}
\newcommand{\tu}{s}

\newcommand{\ome}{\omega_E}

\newcommand{\bevis}{{\bf Proof. }}
\newcommand{\demofin}{\qed \vskip 3.5mm}
\newcommand{\nyp}[1]{\noindent {\bf (#1)}}
\newcommand{\demo}{{\it Proof. }}
\newcommand{\demodone}{\demofin}
\newcommand{\parg}{{\vskip 2mm \addtocounter{theorem}{1}  
                   \noindent {\bf \thetheorem .} \hskip 1.5mm }}

\newcommand{\lcm}{{\text{lcm}}}


\newcommand{\dl}{\Delta}
\newcommand{\cdel}{{C\Delta}}
\newcommand{\cdelp}{{C\Delta^{\prime}}}
\newcommand{\dlst}{\Delta^*}
\newcommand{\Sdl}{{\mathcal S}_{\dl}}
\newcommand{\lk}{\text{lk}}
\newcommand{\lkd}{\lk_\Delta}
\newcommand{\lkp}[2]{\lk_{#1} {#2}}
\newcommand{\del}{\Delta}
\newcommand{\delr}{\Delta_{-R}}
\newcommand{\dd}{{\dim \del}}
\newcommand{\Del}{\Delta}
\newcommand{\PA}{\cS}
\newcommand{\Stbs}{\St \backslash \St_0}

\renewcommand{\aa}{{\bf a}}
\newcommand{\bb}{{\bf b}}
\newcommand{\cc}{{\bf c}}
\newcommand{\xx}{{\bf x}}
\newcommand{\yy}{{\bf y}}
\newcommand{\zz}{{\bf z}}
\newcommand{\mv}{{\xx^{\aa_v}}}
\newcommand{\mF}{{\xx^{\aa_F}}}

\newcommand{\Symm}{\text{Sym}}
\newcommand{\pnm}{{\bf P}^{n-1}}
\newcommand{\opnm}{{\go_{\pnm}}}
\newcommand{\ompnm}{\omega_{\pnm}}

\newcommand{\pn}{{\bf P}^n}
\newcommand{\hele}{{\mathbb Z}}
\newcommand{\nat}{{\mathbb N}}
\newcommand{\rasj}{{\mathbb Q}}

\newcommand{\dt}{\bullet}
\newcommand{\st}{\hskip 0.5mm {}^{\rule{0.4pt}{1.5mm}}}              
\newcommand{\disk}{\scriptscriptstyle{\bullet}}

\newcommand{\cF}{F_\dt}
\newcommand{\pol}{f}

\newcommand{\Rn}{{\mathbb R}^n}
\newcommand{\An}{{\mathbb A}^n}
\newcommand{\frg}{\mathfrak{g}}
\newcommand{\PW}{{\mathbb P}(W)}

\newcommand{\pos}{{\mathcal Pos}}
\newcommand{\g}{{\gamma}}

\newcommand{\Vaa}{V_0}
\newcommand{\Bp}{B^\prime}
\newcommand{\Bpp}{B^{\prime \prime}}
\newcommand{\bbp}{\mathbf{b}^\prime}
\newcommand{\bbpp}{\mathbf{b}^{\prime \prime}}
\newcommand{\bp}{{b}^\prime}
\newcommand{\bpp}{{b}^{\prime \prime}}
\newcommand{\pb}{\overline{p}}
\newcommand{\Pa}{P \backslash \{a \}}
\newcommand{\Min}{\text{Min}}
\newcommand{\Max}{\text{Max}}
\newcommand{\inc}{\text{inc}}
\newcommand{\tih}{\tilde{h}}
\newcommand{\can}{\omega_{\kr[\Delta(\gJ)]}}

\newcommand{\comment}[1]{{\color{blue} \sf ($\clubsuit$ #1 $\clubsuit$)}}

\def\CC{{\mathbb C}}
\def\GG{{\mathbb G}}
\def\ZZ{{\mathbb Z}}
\def\NN{{\mathbb N}}
\def\RR{{\mathbb R}}
\def\OO{{\mathbb O}}
\def\QQ{{\mathbb Q}}
\def\VV{{\mathbb V}}
\def\PP{{\mathbb P}}
\def\EE{{\mathbb E}}
\def\FF{{\mathbb F}}
\def\AA{{\mathbb A}}

\section*{Introduction}

Letterplace and co-letterplace ideals of a partially ordered set $P$
were introduced and studied in \cite{EHM} and \cite{FGH}.
In the latter paper it was shown that many monomial ideals 
studied in the literature derive from letterplace or co-letterplace
ideals as quotients of these ideals by a regular sequence of variable
differences. These ideals therefore allow a powerful unifying treatment
of many classes of ideals.  
\begin{itemize}
\item We present a little used and seemingly little known general method for 
constructing
the resolution of a monomial ideal with linear resolution. It is based 
on a result of K.~Yanagawa \cite[Theorem 4.2]{Yan2000}, 
using the canonical module of a Cohen-Macaulay
Stanley-Reisner ring $\kr[\Delta]$ to give the the linear resolution
of the Alexander dual ideal $I_{\Delta^\vee}$. We discuss in detail how
to compute this canonical module, in particular when  $\Delta$ is a
homology ball, Section \ref{resideals}.
\item This method is used to give a completely explicit and
very simple form of the resolutions of co-letterplace ideals. 
We discuss how the Eliahou-Kervaire resolution appears
as a special case, 
Section \ref{sec:ResCoLp}.
\item As a byproduct we obtain a very large class of simplicial
spheres generalizing Bier spheres, Section \ref{SphereSection}.
\end{itemize}

Given a poset $P$, the $n$'th letterplace ideal $L(n,P)$ is the
monomial ideal generated by monomials
\[ x_{1,p_1}x_{2,p_2} \cdots x_{n,p_n} \]
where $p_1 \leq p_2 \leq \cdots \leq p_n$.
The $n$'th co-letterplace ideal $L(P,n)$ is the monomial ideal generated
by monomials
\[ \Pi_{p \in P} x_{p,i_p}, \]
where $1 \leq i_p \leq n$ and $p < q$ implies $i_p \leq i_q$.
The association $p \mapsto i_p$ gives an isotone map
\[ P \pil \{1 < 2 < \cdots < n \} =: [n]. \]
 The set of these isotone maps, denoted by $\Hom(P,[n])$, 
naturally forms a partially ordered set. If 
$\gJ$ 
is a poset ideal in $\Hom(P,[n])$, we get a subideal
$L(\gJ)$ of $L(P,[n])$, the co-letterplace ideal of the poset ideal. 
In \cite{FGH} these ideals are shown to
have linear quotients and so a linear resolution.

The resolution of $L(P,[n])$ is given already in \cite{EHM}.
In this paper we give explicitly the linear resolution of 
$L(\gJ)$ for any poset ideal $\gJ \sus \Hom(P,[n])$.
Two techniques are commonly used in the literature to construct linear 
resolutions. Either one builds a cellular resolution or, if the ideal has
linear quotients and a regular decomposition function, one can construct an explicit 
resolution by \cite{HeTa}.
Here we advocate and use another general technique, 
building on previous work by Yanagawa \cite{Yan2000, Yan}. A feature of this method is
that it is more of an automatic machine than the two techniques above:
if an ideal $I$ in a polynomial ring
$S$ has linear resolution, 
its Alexander dual ideal $J$ is a Cohen-Macaulay ideal,
and so there is a canonical module $\omega_{S/J}$. The linear resolution
may be constructed directly from this module, as already stated by Yanagawa \cite[Corollary 4.2]{Yan2000}. An explicit knowledge
of the multiplication maps for this module 
enables an explicit description of the differential
in the resolution. In particular, when $I = L(\gJ)$, 
we obtain a simple, explicit description of
this canonical module. It has the somewhat extraordinary 
property that both it and its Alexander dual module identify
(in a multigraded way)
as ideals in (distinct) Stanley-Reisner rings: this is known to happen when $J$ is the 
Stanley-Reisner ring of a homology ball or a homology sphere, which is actually the case here. 
This fact enables a simple explicit form
of the linear resolution of co-letterplace ideals $L(\gJ)$. This in particular
specializes to various explicit linear resolutions in the literature,
like resolutions of strongly stable ideals generated in one
degree \cite{Sin}, \cite{NaRe}, resolutions of uniform face ideals
\cite{Cook}, resolutions of cointerval ideals \cite{Eng}, and 
resolutions of Ferrers ideals, \cite{NaRe}, \cite{NaCo}.

\medskip
A somewhat unintentional consequence of the fact that the canonical
module identifies as an ideal in its Stanley-Reisner ring (respecting
the multigrading)
is that we obtain a very large class of Gorenstein squarefree monomial ideals
of codimension $|P| +1$ giving rise to simplicial spheres.
For each poset $P$, natural number $n$ and nonempty poset ideal $\gJ$
in $\Hom(P,[n])$ we get such a complex. This class generalizes
Bier spheres, studied in \cite{BjZi}, which occur
in the special case when $P$ is an antichain and $n = 2$.

\medskip
The organization of the paper is as follows. In Section \ref{Sec:Pre} we recall
the definitions and basic properties of letterplace and co-letterplace ideals.
In Section \ref{StaircaseSection} we introduce the staircase complex
associated with a poset $P$ and a natural number $n$. This will be pivotal in 
computing our sought for canonical modules. Section \ref{resideals} 
presents the methods we use for computing linear resolutions of monomial
ideal. These methods are particularly nice when the ideals
are Alexander dual to Stanley-Reisner ideals of homology balls or spheres.
In Section \ref{sec:ResCoLp} we use this machinery to give the simple, 
explicit form of the
resolutions of co-letterplace ideals $L(\gJ)$. Section \ref{SphereSection} 
gives the large class of simplicial spheres.

\section{Letterplace and co-letterplace ideals of posets}
\label{Sec:Pre}
We recall the definitions of the above monomial ideals associated with a poset $P$
and a natural number $n$, and give basic properties
of these ideals which we will use. These ideals were introduced in 
\cite{EHM} and \cite{FGH}.

\medskip
Let $\kr$ be a field.
If $R$ is a set denote by $\kr[x_R]$ the polynomial ring 
$\kr[x_i]_{i \in R}$, and if $S \sus R$ denote by $m_S$ the squarefree
monomial $\Pi_{i \in S} x_i$.

If $P$ and $Q$ are finite posets, we denote by $\Hom(P,Q)$ the set
of isotone maps $\phi\colon P \pil Q$, i.e. maps such that $p \leq p^\prime$ implies
$\phi(p) \leq \phi(p^\prime)$. The graph of such a $\phi$ is
\[ \Gamma \phi = \{ (p,\phi(p) \, | \, p \in P \} \sus P \times Q. \]
Let $L(P,Q)$ be the ideal in $\kr[x_{P \times Q}]$ generated by
the monomials $m_{\Gamma \phi}$ with $\phi \in \Hom(P,Q)$. 

Let $[n] = \{1 < 2 < \cdots < n \}$ be the totally ordered set on
$n$ elements. The ideal $L([n],P)$ in $\kr[x_{[n] \times P}]$ is the
{\it $n$'th letterplace ideal} of $P$ and the ideal $L(P,[n])$ in
$\kr[x_{P \times [n]}]$ is the {\it $n$'th
co-letterplace ideal} of $P$. For short, we write these ideals as $L(n,P)$ and
$L(P,n)$ respectively. By 
\cite[Theorem 1.1]{EHM}, see also \cite[Proposition 1.2]{FGH}, 
the ideals $L(n,P)$ and $L(P,n)$ are Alexander dual ideals.

For the convenience of the reader we recall the notion
of Alexander duality.
If $I$ is a squarefree monomial ideal in $\kr[x_R]$, then the
Alexander dual ideal $J$ of $I$ is the monomial ideal consisting of all
monomials in $\kr[x_R]$ which have non-trivial common divisor
with every monomial in $I$. This association is an involution,
so $I$ will
also be the Alexander dual of $J$. We write $J = I^A$. 

\medskip
Note that $\Hom(P,Q)$ is itself a poset, where $\phi \leq \psi$ if and only if $\phi(p) \leq \psi (p)$ for all $p \in P$. Let $\gJ$ be
a poset ideal in the poset $\Hom(P,[n])$, i.e. $\psi \in \gJ$ implies
$\phi \in \gJ$ for all $\phi \leq \psi$. We get a subideal $L(P,n;\gJ)$
of $L(P,n)$
generated by all $m_{\Gamma \phi}$ with $\phi \in  \gJ$. 
In \cite[Proposition 5.1]{FGH} it is shown that $L(P,n;\gJ)$ 
has linear quotients, 
and so has linear resolution.  
(Hence its Alexander dual $L(P,n;\gJ)^A$ is
shellable and so a Cohen-Macaulay ideal.)
Our goal in this paper is to give explicitly the resolution
of $L(P,n;\gJ)$ (for short we write $L(\gJ)$),  
which we do in Section \ref{sec:ResCoLp}.
We mention that in the paper \cite{DFN-LP} we also investigate resolutions
of letterplace ideals $L(n,P)$. In \cite{Katt} homological properties
of the ideals $L(P,Q)$ in general are studied, like regularity, projective
dimension, and length of the first linear strand.

\section{The staircase complex} \label{StaircaseSection}
We introduce a simplicial complex
$\St(P,n)$, the {\it staircase complex} defined by a quadratic monomial
ideal $B(P,n)$, and give the basic properties
of the complex $\St(P,n)$. More generally for a poset ideal $\gJ$
in $\Hom(P,[n])$ we have a subcomplex $\St(\gJ)$ of $\St(P,n)$.

This complex has a central role when describing
the canonical modules of the Alexander duals of co-letterplace ideals.
These canonical modules are again what describes
the resolution of co-letterplace ideals, as we
do in Section \ref{sec:ResCoLp}.

\medskip
Given an isotone map $\phi \in \Hom(P,[n])$ we define a new isotone map
$\phi_-$ given by 
\[ \phi_-(p) = \max \{ \phi(q) \, | \, q < p \}. \]
We let $\max$ of the empty set be $1$, so that $\phi_-(p) = 1$ if
$p$ is minimal. Similarly we let $\phi^+$ be given by 
\[ \phi^+(p) = \min \{ \phi(q) \, | \, p < q \} \]
where $\min$ of the empty set is $n$, so that $\phi^+(p) = n$ if
$p$ is a maximal element of $P$. 

\begin{proposition}
The maps 
\[ \Hom(P,[n]) \bihom{ (-)_-}{(-)^+} \Hom(P,[n])\]
form a Galois correspondence between posets.
\end{proposition}

\begin{remark}
If we consider $\Hom(P,[n])$ as a category, this simply
says that the above is an adjunction, with $(-)_-$ the left adjoint.
\end{remark}

\begin{proof}
We need to show that $\phi \leq \psi^+ $ if and only if $\phi_- \leq \psi$.
So suppose $\phi \leq \psi^+$. Fix $p$. Then for every $q < p$ we have
\[ \phi(q) \leq \psi^+(q) \leq \psi (p). \]
Hence $\phi_-(p) \leq \psi(p)$. The other direction is similar.
\end{proof}

We define the {\it upper normalization} of the isotone map $\phi$ to be 
$(\phi_-)^+$. 
If $\phi$ equals this normalization  we say $\phi$ is {\it upper normal}.
One may also define the {\it lower normalization} of $\phi$ to be
$(\phi^+)_-$.

\medskip
If $\phi$ is an isotone map let
\[ T_*(\phi) = \{ (p,i) \, | \, \phi(q) \leq i \leq \phi(p) 
\text{ for all } q < p \}. \]
Similarly let 
\[ T^*(\phi) = \{ (p,i) \, | \, \phi(p) \leq i \leq \phi(q) 
\text{ for all } q > p \}. \]

\begin{remark} \label{rem:Path-nm} If $P = [m]$ is totally ordered, then $\phi$ 
is an upper normal isotone
map in $\Hom([m], [n])$ if and only if $\phi(m) = n$. 
These are again precisely those isotone maps such that $T_*(\phi)$ 
is a lattice path from $(1,1)$ to $(m,n)$, i.e. a maximal chain
in $[m] \times [n]$. Similarly the lower normal isotone maps
are precisely those such that $\phi(1) = 1$ and they again correspond
precisely to those $\phi$ such that $T^*(\phi)$ is a lattice path
from $(1,1)$ to $(m,n)$. 
\end{remark}


\begin{proposition} Let $\phi$ be an upper normal map
and $\psi$ a lower normal map corresponding to each other by
the Galois correspondence.
That is, $\psi = \phi_- $ and $\psi^+ = \phi$. Then
\[T^*(\psi) = T_* (\phi).\]
\end{proposition}

\begin{proof}
We show for any isotone map $\eta$ that
i) $T_*(\eta) \subseteq T^*(\eta_-)$ and ii)  $T^*(\eta) \subseteq T_*(\eta^+)$.
This gives
\[ T_*(\phi) \sus T^*(\phi_-) \sus T_*((\phi_-)^+) = T_*(\phi). \]
Since $\phi_- = \psi$ the claim follows. We show only i) since ii) 
is analogous.
Suppose $(p,i)$ is in $T_*(\eta)$. By definition for all $q<p$ we have 
$\eta(q)\leq i\leq \eta(p)$.
This shows that $\eta_-(p)\leq i$. If $q'>p$ then 
$\eta_-(q')$ is $\max\{\eta(r)|r<q'\}$ and this is $\geq \eta(p)\geq i$.
Therefore for all $q'>p$ we have $\eta_-(p)\leq i\leq \eta_-(q')$. 
Hence $(p,i)$ is in $T^*(\eta_-)$.
\end{proof}

\begin{definition} Let $B(P,n)$ be the squarefree monomial  ideal
in $\kr[x_{P \times [n]}]$ generated by the quadratic monomials 
$x_{p,i}x_{q,j}$ such that $p < q$ and $i > j$. 
(Thus $B(P,n)$ is generated by pairs $(p,i) < (q,j)$ in 
$P \times [n]^{\text{op}}$ 
which are bistrict, i.e. strictly bigger in each coordinate.)
Let $\St(P,n)$ be the simplicial complex associated with this
squarefree monomial ideal. We call it the $n$'th {\it staircase complex}
(inspired by Remark \ref{rem:Path-nm} above).

The staircase complex is a simplicial complex on the vertex set $P \times [n]$.
Let $\St_0(P,n)$ be the subcomplex of $\St(P,n)$ consisting
of all faces $F \sus P \times [n]$ such that the projection down to $P$
is not surjective. Let $\Stbs(P,n)$ be the relative complex
$\St(P,n) \backslash \St_0(P,n)$. 
\end{definition}

\begin{proposition} \label{pro:StairFacet}
 We consider isotone maps $\phi$ in $\Hom(P,[n])$.
\begin{itemize}
\item[a.] $T_* \phi$ is a face of $\St(P,n)$ for each isotone map $\phi$.
\item[b.] Every face of $\St(P,n)$ is contained in $T_* \phi$ for some
$\phi$.
\item[c.] The maximal faces of $\St(P,n)$ are precisely the $T_* \phi$ where
$\phi$ are upper normal isotone maps.
\item[d.] The relative complex $\Stbs(P,n)$ is the disjoint union
of intervals $[\Gamma \phi, T_* \phi]$ as $\phi$ ranges over the
isotone maps.
\end{itemize}
\end{proposition}

\begin{proof}

a. Let $(q,j)$ and $(p,i)$ be in $T_* \phi$ with $q < p$. Then 
$j \leq \phi(q) \leq i \leq \phi(p)$. Therefore $T_*\phi$ is in $\St(P,n)$.

b. Let $F$ be a face of $\St(P,n)$. 
Consider the projection $p_1(F)$ of $F$ on $P$. 

For each $p$ in $p_1(F)$, 
let $\alpha(p) = \max \{ i \, | \, (p,i) \in F \}$. 
Then clearly $q < p$ in $p_1(F)$ implies $\alpha(q) \leq \alpha(p)$. For each 
$p$ not in $p_1(F)$, define 
\[ \alpha(p) = \max \{ \alpha(q) \, | \, q < p \text{ and } q 
\text{ is in } p_1(F) \}. \]
This makes $\alpha$ an isotone map defined on $P$. 
Also if $(p,i)$ is in $F$, then
for every $q < p$ with $q$ in $p_1(F)$, we must have $\alpha(q) \leq i$. 
Hence $\alpha(q) \leq i \leq \alpha(p)$ for {\it every} $q < p$
regardless of whether $q$ is in $p_1(F)$ or not. 
Thus $F$ is contained in $T_* \alpha$.

c. Let $(p,i) \in T_* \phi$. Then $\phi(q) \leq i \leq \phi(p)$
for all $q < p$. We have
\begin{align*}
\phi_-(p) & =  \max \{ \phi(q) \, | \, q < p \} \leq i & \\
(\phi_-)^+(q)  & = \min \{ \phi_-(r) \, | \, q < r \} \leq \phi_-(p).& 
\end{align*} 
Also 
\[ (\phi_-)^+(p) = \min \{ \phi_-(q) \,| \, p < q \} \geq \phi(p) \geq i \]
and so
\[ (\phi_-)^+(q) \leq i \leq (\phi_-)^+(p).\]
This gives $(p,i) \in T_*((\phi_-)^+)$.

Suppose that $\phi$ is not upper normal, i.e. $\phi < (\phi_-)^+$. 
Let $p$ be such that $\phi(p) < (\phi_-)^+(p) =: i$. Then
$(p,i)$ is in $T_*((\phi_-)^+)$ but not in $T_* \phi$ and so the
latter is not a facet.

d. Let $F$ be a face whose projection on $P$ is surjective.
Construct the map $\alpha$ as in part b. Then clearly 
$\Gamma \alpha \sus F$ and so $F$ is in the interval determined by 
$\alpha$, by the end of the argument of b. 
Thus the union of the $[\Gamma \phi, T_* \phi]$ is all of 
$\Stbs(P,n)$.
If  $\Gamma \phi \sus F \sus T_* \phi$, then
$(p,\phi(p)) \in F$ and 
$\phi(p) = \max \{ i \, | \, (p,i) \in F \}$. Thus $\phi$ is uniquely
determined from $F$, and so $F$ is in only one such interval.
\end{proof}

\begin{definition}
If $\gJ$ is a poset ideal in $\Hom(P,n)$, the complement $\gJ^c$
is a poset filter in $\Hom(P,n)$. It gives rise to a subideal
$L(\gJ^c)$ of $L(P,n)$, the ideal generated by all $m_{\Gamma \phi}$ 
where $\phi$ is in $\gJ^c$. 

  Let $B(\gJ)$ be the ideal $L(\gJ^c) + B(P,n)$ and
$\St(\gJ)$ the associated simplicial complex of this squarefree
monomial ideal, the {\it staircase complex associated with $\gJ$}. 
Note that $\St_0(P,n)$ is a subcomplex of 
$\St(\gJ)$ and so we can consider the relative complex
$\St(\gJ) \backslash \St_0(P,n)$.
\end{definition}

\begin{corollary} \label{StairFacetPosetIdeal}
The relative complex $\St(\gJ) \backslash \St_0(P,n)$ is the disjoint union
of intervals $[\Gamma \phi, T_* \phi]$ as $\phi$ ranges over the
isotone maps in $\gJ$.
\end{corollary}

\begin{proof}
Let $\phi$ and $\psi$ be maps in $\gJ$ and $\gJ^c$ respectively.
The task is to show that $m_{T_* \phi}$ is not divisible by $m_{\Gamma \psi}$
or equivalently that $\Gamma \psi$ is not contained in $T_* \phi$. 
Since $\psi$ is not $\leq \phi$ there is some $p \in P$ with 
$j = \psi(p) > \phi(p)$. But then $(p,j)$ is in $\Gamma \psi$ and not
in $T_* \phi$. 
\end{proof}

\section{Squarefree ideals with linear resolution} \label{resideals}

For ideals with linear resolution, explicit resolutions have been
given in a number of places in the literature. Two approaches are common.
One is to give a cellular resolution. Another is when the ideal has linear
quotients and an extra 
condition is fulfilled, i.e. there is a regular decomposition function, 
see \cite{HeTa}.
The typical example of the latter is the Eliahou-Kervaire resolution.

 Here we advocate a third approach, using the canonical module.
Let $I$ be a monomial
ideal with linear resolution in a polynomial ring $S$, and $J$ the Alexander dual ideal.
Then $S/J$ is a Cohen-Macaulay ring by \cite{EaRe}, and so has
a canonical module $\omega_{S/J}$. Firstly, we recall a result of Yanagawa
\cite[Thm.4.2]{Yan2000} that 
the resolution of $I$ may
be constructed directly from $\omega_{S/J}$.  
Secondly we show how to get hold of the canonical module $\omega_{S/J}$.

The general methods here can be found in the literature, 
and we give 
appropriate pointers. However, they seem little known and have been little used,
so we think it of value to present them. In Section 4 we use
them very effectively to construct resolutions.
It also seems that the approach here may be used in most cases in the literature
where an explicit linear resolution of a monomial ideal 
has been given. 

\subsection{Squarefree modules and linear resolutions}
Let $A$ be a finite set. We consider the polynomial ring $\kr[x_A]$.
Let $\NN^A = \Hom(A,\NN)$. If $\bb$ is in $\NN^A$ write $b_i$ 
for $\bb(i)$ where $i \in A$. Let $\ep_i$ be the unit coordinate vector
with $\ep_i(j) = \delta_{ij}$ for $i,j$ in $A$. 
Let $M$ be a finitely generated $\NN^A$-graded module over the polynomial ring 
$\kr[x_A]$. The module $M$ is {\it squarefree} provided
for each multidegree $\bb \in \NN^A$ the multiplication map
\[ M_\bb \mto{x_i} M_{\bb + \ep_i} \]
is an isomorphism whenever $b_i \geq 1$. This means that $M$ is
essentially determined by the graded parts $M_\bb$ where $\bb$ is
a multidegree with $0$'s and $1$'s, and the multiplication maps
between these multigraded parts.

Let $\bfen$ be the multidegree $\sum_{i \in A} \ep_i$. 
The {\it Alexander dual} of $M$ is the squarefree module $M^*$
where for $\bb \in \{0,1 \}^A$ the multigraded part $(M^*)_\bb$ equals  
$\Hom_\kr(M_{\bfen-\bb}, \kr)$. If $b_i = 0$, the multiplication 
\[ (M^*)_\bb \mto{x_i} (M^*)_{\bb + \ep_i} \]
is the dual of the multiplication
\[M_{\bfen-\bb-\ep_i} \mto{x_i} M_{\bfen-\bb}. \]
By obvious extension, this defines $(M^*)_\bb$ for all $\bb \in \NN^A$ and
all multiplications. The functor obtained in this way is exact and contravariant.

If $R$ is a subset of $A$, its indicator
vector  $\ep_R = \sum_{i \in R} \ep_i$ 
is a $0,1$-vector in $\NN^A$. We shall often write $M_R$
instead of $M_{\ep_R}$. Also let $R^c$ be the complement $A \backslash R$. 

Note that if $I$ is a squarefree monomial ideal, then the 
Alexander dual {\it module} of $I$ is $S/I^A$, where $I^A$
is the Alexander dual {\it ideal} of $I$.

If $R \sus A$ we denote by $S(-R)$ the free squarefree $S$-module with 
generator $e_R$ of degree $\ep_R$. It identifies naturally as the ideal
in $S$ generated by the monomial $m_R$. Its Alexander dual module is
$S/(x_i)_{i \in R}$. 

\medskip
Given a squarefree module K.~Yanagawa \cite[p.~297]{Yan} defines a complex
of free squarefree modules $\gL(M)$ whose terms are
\[ \gL_i(M) = \bigoplus_{|R^c| = i} (M_R)^\circ \te_{\kr} S, \]
where $(M_R)^\circ$ is $M_R$ but considered to have multidegree $R^c$. 
Put a total order on $A$. The differential is
\[ m^\circ \te S \mapsto \sum_{j \in R^c} (-1)^{\alpha(j,R)}(x_jm)^\circ \te x_js \]
where $\alpha(j,R)$ is the number of $i \in R$ such that $i < j$.

The complex $\gL(M^*)$ associated with the Alexander dual of $M$
will then be the dual complex $\Hom_S(\gL(M), \omega_S)$,
where $\omega_S = S(-\bfen)$ is the canonical module of $S$.

\begin{proposition} \label{SqlinProMN}
A squarefree module $M$ is Cohen-Macaulay if and only if the Alexander dual
$M^*$ has linear resolution.
In this case let $M$ have codimension $d$,
and let $N = \Ext^d(M,\omega_S)$ be the dual module.
(This is also a Cohen-Macaulay squarefree module of codimension $d$.)

Then $\gL(N)$ is a resolution of the Alexander dual $M^*$. In
particular $M^*$ has $d$-linear resolution.
\end{proposition}

\begin{proof}
This follows by \cite[Theorem 3.8]{Yan}.
The functor $\gL$ is there denoted by $\gF$, and the resolution of
$N^*$ is denoted by $\gP^\dt$. 
\end{proof}


\medskip

Now let $I$ be a squarefree monomial ideal with linear resolution and let 
$J$ be the Alexander dual {\it ideal}. (Note that the Alexander dual of the 
squarefree {\it module} $I$ is $I^* = S/J$.) Then $J$ is 
a Cohen-Macaulay ideal by the above proposition and so has a squarefree
canonical module $\omega_{S/J}$. The following enables 
one to give the explicit resolution of $I$ if one has a sufficiently 
explicit description of the canonical module. 

\begin{proposition}[\cite{Yan2000}] \label{prop:SqLinRes} 
Let $I$ be an ideal with linear resolution,
and $J$ the Alexander dual ideal of $I$.
The resolution of $I$ is $\gL(\omega_{S/J})$. 
\end{proposition}

\begin{proof} 
This follows from Proposition \ref{SqlinProMN} by letting $M^*$ be $I$.
\end{proof}

\begin{remark}
The dual complex $\Hom_S(\gL(\omega_{S/J}), \omega_S)$ is $\gL(\omega^*_{S/J})$, so 
describing the resolution (i.e. the terms and differentials between them)
may equally well rely on a good description of $\omega^*_{S/J}$. 
\end{remark}

\subsection{Describing the canonical module I: A general approach} 
\label{describingcanonical}
From the above proposition, if 
we have a good description of $\omega_{S/J}$ or its 
Alexander dual, we get a good description of the linear resolution 
of $I$. We now address how to describe $\omega_{S/J}$. 
Various sources do this. Chapter 5.7 in \cite{BrHe} is particularly 
useful and is our main reference here. First we describe a general
approach.
Let $\Delta$ be a simplicial complex on the set of vertices $A$ 
and $\kr[\Delta]$ its
Stanley-Reisner ring. The {\it enriched cochain complex} of $\Delta$
is $\gL(\kr[\Delta])$. If $\Delta$ has dimension $d$, this complex
is
\begin{equation} \label{SqlinLigResEnrich}
 \bigoplus_{\overset{|F| = d+1}{F \in \Delta}}
S(-F^c) \vpil
\bigoplus_{\overset{|F| = d} {F \in \Delta}} S(-F^c)
\vpil \cdots \vpil S(-\bfen).
\end{equation}
One may think of this as ``fattening'' up the standard augmented 
cochain complex of $\Delta$, see \cite[Definition 1.20]{StMi},
a complex of $\kr$-vector spaces, so
that it becomes a complex of $S$-modules. More generally the enriched
cochain complex (and chain complex) can be defined for any cell complex,
see \cite{FlCMCell}.

Here is a way to get hold of the canonical module.

\begin{proposition} \label{prop:SqLinCan}
Let $\Delta$ be a Cohen-Macaulay simplicial complex. Then 
the enriched cochain complex $\gL(\kr[\Delta])$ is a resolution of
$(\omega_{\kr[\Delta]})^*$, the Alexander dual of the canonical module.
\end{proposition}

\begin{proof}
This follows from Proposition \ref{SqlinProMN} by letting 
$M = \omega_{\kr[\Delta]}$ and $N = \kr[\Delta]$. 
\end{proof}

Let $P_F$ be the prime ideal generated by 
$(x_i \, | \, i \in F^c)$. The Alexander dual of $S(-F^c)$ is 
$S/P_F$.
If we take the Alexander dual of the complex (\ref{SqlinLigResEnrich}), we get an injective resolution  of $\omega_{\kr[\Delta]}$ in  
the category of squarefree modules:
\[ \omega_{\kr[\Delta]} \pil 
\bigoplus_{\overset{|F| = d+1}{F \in \Delta}}  S/P_F
\pil 
\bigoplus_{\overset{|F| = d}{F \in \Delta}}  S/P_F
\pil \cdots \pil S/{\mathfrak m}. \]
where $\mathfrak{m}$ is the homogeneous maximal ideal.
This resolution is given in \cite[Theorem 5.7.3]{BrHe}.

\subsection{Describing the canonical module II: Homology balls} 
\label{describingcanonical2}

When $\Delta$ is topologically a ball, the canonical module $\omega_{k[\Delta]}$
can be described explicitly as an ideal in the Stanley-Reisner ring 
$k[\Delta]$. This will be the case in our applications. We can even describe 
in full generality when the canonical module identifies
as an ideal in a multigraded way. We rely on Chapter 5.7 in \cite{BrHe}.
First we give necessary definitions.


\begin{definition}
Let $\Delta$ be a simplicial complex of dimension $d$ and let $\kr$ be a field. We say $\Delta$ is a \emph{homology sphere} over $\kr$ if, for every face $F \in \Delta$, one has that $\lk_{\Delta}F$ has the (reduced) homology over $\kr$ of a $\dim(\lk_{\Delta}F)$-dimensional sphere. In other words,
\[ \tH_i(\lk_{\Delta}F, \kr) \cong \begin{cases} \kr & i = \dim(\lk_{\Delta}F) \\
                               0 & i \neq \dim(\lk_{\Delta}F)
                   \end{cases}
\]
for every face $F$ of $\Delta$.
\end{definition}

\begin{remark}
Note that any simplicial complex which is topologically a sphere will
be a homology sphere.
Homology spheres are also known as \emph{Gorenstein*} complexes, i.e. complexes which are not cones and whose Stanley-Reisner ring is Gorenstein. Equivalently, homology spheres are the same as Cohen-Macaulay Eulerian complexes: see \cite[Chapter 5]{BrHe} for more details. 
\end{remark}

\begin{definition}
Let $\Delta$ be a simplicial complex of dimension $d$ and let $\kr$ be a field. We say $\Delta$ is a \emph{homology ball} over $\kr$ if there exists a subcomplex $\Sigma$ of $\Delta$ such that:
\begin{itemize}
\item $\Sigma$ is a $(d-1)$-dimensional homology sphere over $\kr$;
\item $\lk_{\Delta}F$ has the homology of a $\dim(\lk_{\Delta}F)$-dimensional sphere if $F \notin \Sigma$ and otherwise has zero homology.
\end{itemize}
Note that the facets of $\Sigma$ are those codimension $1$ faces of $\Delta$
which are contained in precisely one facet of $\Delta$. We therefore call $\Sigma$
the $\emph{boundary}$ of $\Delta$, and write $\partial\Delta = \Sigma$.
\end{definition}

\begin{remark} Any simplicial complex which is topologically a
ball will be a homology ball.
\end{remark}


The following gives a complete description of when the canonical module
of a simplicial complex, identifies in a multigraded way as an ideal
in the corresponding Stanley-Reisner ring.

\begin{theorem} \label{SRcanonical}
Let $\Delta$ be a Cohen-Macaulay simplicial complex on $n$ vertices. Then the canonical module $\omega_{\kr[\Delta]}$ can be seen as a $\ZZ^n$-graded ideal of $\kr[\Delta]$ precisely when $\Delta$ is either a homology sphere or a homology ball. More precisely:
\begin{itemize}
\item $\omega_{\kr[\Delta]} \cong \kr[\Delta]$ as $\ZZ^n$-graded $\kr[\Delta]$-modules if and only if $\Delta$ is a homology sphere;
\item $\omega_{\kr[\Delta]}$ is a proper $\ZZ^n$-graded ideal of $\kr[\Delta]$ if and only if $\Delta$ is a homology ball. In this case, $\omega_{\kr[\Delta]} \cong \bar{I}_{\partial\Delta}$, where $\bar{I}_{\partial\Delta}$ is the image inside $\kr[\Delta]$ of $I_{\partial\Delta}$, the Stanley-Reisner ideal of the boundary of $\Delta$.\end{itemize}
\end{theorem}

\begin{remark}
Theorem 5.7.1 in \cite{BrHe} actually says that the canonical module of a homology ball $\Delta$ identifies as the image inside $\kr[\Delta]$ of the monomial ideal $\mathfrak{a} = (x^F \mid F \in \Delta, F \notin \partial\Delta)$; however, it is easy to see that $\mathfrak{a} + I_{\Delta} = I_{\partial\Delta}$, whence our reformulation.
\end{remark}

Whenever $\Delta$ is either a homology ball or a homology sphere, both $\omega_{\kr[\Delta]}$ and its dual $\omega_{\kr[\Delta]}^*$ can be seen as $\mathbb{Z}^n$-graded ideals inside some Stanley-Reisner ring. More precisely, since the following sequence
\begin{equation} \label{boundarysequence}
0 \mapsto I_{\Delta} \to I_{\partial\Delta} \to \kr[\Delta] \to \kr[\partial\Delta] \to 0
\end{equation}
is exact, by applying the Alexander duality functor we get another exact sequence
\[0 \vpil \kr[\Delta^A] \vpil \kr[(\partial\Delta)^A] \vpil I_{\Delta}^A \vpil I_{\partial{\Delta}}^A \vpil 0\]
and, since by Theorem \ref{SRcanonical} $\omega_{\kr[\Delta]}$ identifies as the image of $I_{\partial\Delta}$ inside $\kr[\Delta]$, we get that $\omega_{\kr[\Delta]}^*$ identifies as the image of $I_{\Delta}^A$ inside $\kr[(\partial\Delta)^A]$.

The sequence \eqref{boundarysequence} is only one of the possible liftings of the exact sequence of $\kr[\Delta]$-modules

\begin{align} \label{eq:LinresCanBound}
0 \to \omega_{\kr[\Delta]} \to \kr[\Delta] \to \kr[\partial\Delta] \to 0
\end{align}
(where all maps are the natural ones). One may use other presentations of the
$S$-module $\omega_{\kr[\Delta]}$, as we shall do later. We are then modifying 
the two leftmost objects of \eqref{boundarysequence}, but not the two rightmost
ones. In particular, $\kr[\partial\Delta]$ will appear in the sequence for 
any choice of a representative of $\omega_{\kr[\Delta]}$. We use this
in the discussion before Theorem \ref{explicitboundary}.

\medskip
To recover the notation of the beginning of this section, let $I$ be a squarefree monomial ideal with linear resolution and let $J$ be its Alexander dual ideal.
By Theorem \ref{SRcanonical}, we have a complete understanding of the canonical module $\omega_{S/J}$ when $J$ is the Stanley-Reisner ideal of either a homology sphere or a homology ball. The multigraded structure is completely characterized by knowing the faces of the simplicial complex and, in the homology ball case, of its boundary. This implies that in these cases we can find a complete combinatorial description (maps included!) of the minimal free resolution of $I$. 


\begin{theorem} \label{specialres}
Let $\Delta$ be a $d$-dimensional Cohen-Macaulay simplicial complex and let $I$ denote the ideal which is Alexander dual to the Stanley-Reisner ideal of $\Delta$. Then:
\begin{itemize}

\item if $\Delta$ is a homology sphere, then $I$ is minimally resolved by $\gL(\kr[\Delta])$, the enriched cochain complex (see Subsection \ref{describingcanonical});

\item if $\Delta$ is a homology ball, let $\bar{I}_{\partial\Delta}$ be the image inside $\kr[\Delta]$ of $I_{\partial\Delta}$, the Stanley-Reisner ideal of the boundary of $\Delta$. Then $I$ is minimally resolved by $\gL(\bar{I}_{\partial\Delta})$, i.e. the following complex:
\[ \bigoplus_{\overset{|F| = d+1}{F \in \Delta,~F \notin \partial\Delta}}
S(-F^c) \vpil
\bigoplus_{\overset{|F| = d} {F \in \Delta,~F \notin \partial\Delta}} S(-F^c)
\vpil \cdots \vpil \bigoplus_{\overset{v~\text{vertex of}~\Delta}{v \notin \partial\Delta}} S(-v^c) \vpil 0.\]
\end{itemize}
\end{theorem}

\begin{proof}
This is a direct consequence of Proposition \ref{prop:SqLinRes} and Theorem \ref{SRcanonical}.
\end{proof}

\begin{remark} Yanagawa \cite[Remark 4.4.(a)]{Yan2000} states this result
under a more restrictive hypothesis, namely requiring the geometric
realization of $\Delta$ to be a manifold with Gorenstein* boundary.
That the canonical module $\omega_{\kr[\Delta]}$ then identifies as a 
$\hele^n$-graded ideal in $\kr[\Delta]$ is an (unpublished) result of Hochster,
see \cite[Theorem 5.7.2]{BrHe}.
\end{remark}

Note that we do not need any shellability hypothesis in the theorem above.
\begin{example}
Let $\Delta$ be the non-shellable 3-dimensional homology ball on 9 vertices and 18 facets constructed by F.~H.~Lutz in \cite{LutzBall}. The complete $f$-vector is (1, 9, 33, 43, 18).

Macaulay2 \cite{Macaulay2} computations show that the boundary of this ball has $f$-vector (1, 9, 21, 14).

Then one checks again by Macaulay2 that the resolution of $I$ is 
\[S(-5)^{18} \vpil S(-6)^{29(=43-14)} \vpil S(-7)^{12(=33-21)} \vpil 0,\]
as prescribed by Theorem \ref{specialres}. One may of course recover the maps and the multigraded structure of the resolution as well.
\end{example}

\begin{remark}
When $\Delta$ is a doubly Cohen-Macaulay module, or equivalently $\kr[\Delta]$ 
is a level Cohen-Macaulay ring wih nonzero $\tH_{d-1}(\Delta;\kr)$,
then $\omega_{\kr[\Delta]}$ can be described quite 
explicitly also, see Remark 4.4.(b) in \cite{Yan2000} and
Corollary 5.7 in \cite{BrHe}.
\end{remark}

\section{Resolutions of co-letterplace ideals}
\label{sec:ResCoLp}

In \cite[Proposition 5.1]{FGH} it is shown that the co-letterplace ideal 
$L(P,n;\gJ)$
for a poset ideal $\gJ \sus \Hom(P,[n])$ has
a linear resolution. In \cite[Theorem 3.6]{EHM} an explicit form of the linear
resolution of the co-letterplace ideal $L(P,n)$ is given,
by using that $L(P,n)$ has linear quotients and a regular
decomposition function. We also mention that
\cite{Bol} gives a recursive formula for the Betti numbers
of $L(P,n)$ in terms of the Betti numbers of $L(\alpha,n-1)$
as $\alpha$ varies over poset ideals in $P$.

Here we give more generally the explicit form of  the linear
resolution of $L(P,n;\gJ)$ with our approach via the canonical module,
Proposition \ref{prop:SqLinRes}, using Proposition \ref{prop:SqLinCan}
to describe the canonical module explicitly. In \cite{FGH} the second
author et al.~show
that many ideals with linear resolution are regular quotients by
variable differences of 
co-letterplace ideals $L(P,n;\gJ)$. In the literature many of these
resolutions are given in explicit form: \cite{Sin} and 
\cite{NaRe} contain cellular resolutions of strongly stable ideals, and
polarizations thereof, \cite{NaCo} and \cite{NaRe} give
cellular resolutions of Ferrers ideals, \cite{Cook} gives 
cellular resolutions of uniform face ideals, and \cite{Eng} gives cellular
resolutions of cointerval ideals.

A consequence of our result is that we obtain these resolutions
in a new, unified, very simple form,
and our situation is also considerably more general. We give a more specific
discussion of this for resolutions of strongly stable ideals in Subsection 
\ref{subsec:ResCoLpStSt}.

\subsection{The dual of the canonical module}
\label{subsec:DualCanonical}
Since $L(P,n;\gJ)$ has linear resolution, 
the Alexander dual $L(P,n;\gJ)^A$, which we denote as 
$L(\gJ;n,P)$, is a Cohen-Macaulay ideal.
An explicit description of this ideal, which has codimension $|P|$, is given in \cite[Theorem 5.8]{FGH}.

\begin{definition}
We denote by $\Delta(\gJ;n,P)$ (or for short $\Delta(\gJ)$) 
the simplicial complex whose
Stanley-Reisner ideal is the Alexander dual $L(\gJ;n,P)$. 
This is then a Cohen-Macaulay simplicial complex, 
in fact shellable, since $L(\gJ) = L(P,n;\gJ)$ has linear quotients.
\end{definition}

Let $\omega_{\kr[\Delta(\gJ)]}$ be the canonical module of the 
Stanley-Reisner ring $\kr[\Delta(\gJ)]$.
We now describe this canonical module, 
or rather its Alexander dual module.
Recall that $B(\gJ) = L(\gJ^c) + B(P,n)$ where  
$B(P,n)$ is the ideal generated by the quadratic monomials
$x_{p,i} x_{q,j}$ where $p < q$ and $i > j$. We define $L(\gJ|\St)$ to be
the image of the composition
\[ L(\gJ) \pil \kr[x_{P \times [n]}] \pil \kr[x_{P \times [n]}]/B(\gJ), \]
a subquotient of $\kr[x_{P \times [n]}]$. Here is our simple description
of the canonical module $\omega_{\kr[\Delta(\gJ)]}$.

\begin{proposition} \label{pro:ResCoLpCanD}
The canonical module $\omega_{\kr[\Delta(\gJ)]}$ is the Alexander dual module
of $L(\gJ|\St)$.
\end{proposition}

\begin{proof} The complex $\Delta(\gJ)$
has facets which are the complements of the 
$\Gamma \phi$, where $\phi \in \gJ$ (this is a general fact of Alexander 
duality). 
By Proposition \ref{prop:SqLinCan}  the resolution of $(\can)^*$ is the enriched cochain complex
\[ \bigoplus_{\phi \in \gJ} S(-\Gamma \phi) 
\vmto{d} \bigoplus_{\overset{|T| = |P| + 1}{T \supseteq \,
\text{some} \, \Gamma \phi, \phi \in \gJ}} S(-T) \longleftarrow  \cdots . \]

We describe the map $d$. Let $e_T$ be the generator of $S(-T)$.
Note that given such a $T$ to the right above, there is a unique
$p \in P$ with two distinct elements $(p,i), (p,j)$ in $T$ with 
$i < j$.

\medskip
1.  If the monomial $m_T$ is in the ideal $B(\gJ)$, 
then $T$ contains $\Gamma \phi$ 
for a unique $\Gamma \phi$ by the above comment. 
Suppose $(p,j)$ is in $T$ but not in $\Gamma \phi$.
Then $e_T \mapsto \pm x_{p,j} e_{\Gamma \phi}$. 

\medskip
2. Suppose the monomial $m_T$ is not in the ideal $B(\gJ)$.
Then $T\backslash \{(p,j) \}$ is a graph  $\Gamma \phi$, while
$T \backslash \{(p,i) \}$ is a graph $\Gamma \psi$: note that both $\phi$ and $\psi$ lie in $\gJ$. The map $d$ sends
\[ e_T \mapsto \pm (x_{p,j} e_{\Gamma\phi} - x_{p,i} e_{\Gamma \psi}). \]
Consider the map 
\[ \oplus_{\phi \in \gJ} S(- \Gamma \phi) \mto{p} S \]
sending the generator $e_{\Gamma \phi}$ to $m_{\Gamma \phi}$. We see that if 
$T$ is as in case 1. above, then  $d(e_T)$ maps by $p$ to the ideal 
$B(\gJ)$. 
If $T$ is as in case 2. then $d(e_T)$ maps to zero by $p$.
Hence we get a complex
\[ S/B(\gJ) \vmto{\pb} \bigoplus_{\phi \in \gJ} S(-\Gamma \phi) 
\vmto{d} \bigoplus_{\overset{|T| = |P| + 1}{T \supseteq \,
\text{some}\, \Gamma \phi}, \, \phi \in \gJ} \vpil  \cdots . \]
The image of $\pb$ is exactly $L(\gJ|\St)$. We will show that
the complex above is exact in the middle displayed term.
This will prove our statement.
\medskip

So suppose $r = \sum_\phi r_\phi e_{\Gamma \phi}$ maps to zero in $S/B(\gJ)$.
We may assume the sum is homogeneous for the multigrading, and
even squarefree and so the $r_\phi$ are monomials. 
Let  $R \sus P \times [n]$ be the subset corresponding
to its multidegree.



\medskip
1. Suppose some $x_{p,i}$ in $r_\phi$ is such that $x_{p,i} m_{\Gamma \phi}$
is in $B(\gJ)$. Then if $T$ is the degree of $x_{p,i} m_{\Gamma \phi}$ we can 
subtract a suitable multiple of the image of $d(e_T)$ from $r$. 
This will remove the term $r_\phi e_
{\Gamma \phi}$ and add no other terms.

\medskip
2. Suppose 1. does not apply. 
Suppose there is some $(p,i)$ in $r_\phi$ with $i < \phi(p)$.
Let $\phi^\prime$ be the map $\phi$ but with
the change that $p$ is sent to $i$. This is an isotone map since if 
$q < p$ we cannot have $\phi(q) > i$ as 1. does not apply.
Hence $\phi(q) \leq i$. 
We see that $\phi^\prime < \phi$ and hence $\phi' \in \gJ$. Let $T \sus P \times [n]$ give the
degree of $x_{p,\phi(p)}m_{\Gamma \phi^\prime}$, which is also the degree 
of $x_{p,i} m_{\Gamma \phi}$. We see that 
\[  \frac{r_\phi}{x_{p,i}}e_T \overset{d}\mapsto 
 \pm \frac{r_\phi}{x_{p,i}} x_{p,\phi(p)} e_{\Gamma \phi^\prime}
\mp  r_\phi e_{\Gamma \phi}.\]
Hence we may add a multiple of this to $r$ and get terms involving
smaller $\phi$'s.

3. Suppose every $x_{p,i}$ in
every $r_\phi$ is such that $i > \phi(p)$. (We cannot have equality since $r$
has squarefree degree).  But then it is easily seen that $\phi$ is 
uniquely determined by the squarefree degree of $r$, and so there can
only be one term in $r$. But then $r$ cannot map to zero in $L(\gJ|\St)$.

The upshot is that as long as $r$ is nonzero, we can always
adjust with a boundary to get that the terms
involve smaller and smaller $\phi$'s. In the end we will be left with
$r = 0$. Thus we have shown that $(\can)^*$ is $L(\gJ|\St)$ and
hence $\can$ is the Alexander dual $L(\gJ|\St)^*$. 
\end{proof}

As noted after Theorem \ref{SRcanonical}, a consequence of the above is that
both the canonical module $\can$ and its Alexander dual identify
as ideals in Stanley-Reisner rings in a multigraded way.

\begin{corollary} \label{cor:ResCoLpBoth}
The Alexander dual module $(\can)^*$ is the image of the natural map 
\[ L(\gJ) \pil \kr[x_{P \times [n]}]/ B(\gJ) \]
and $\can$ is the image of the natural map 
\[ B(\gJ)^A \pil   \kr[x_{P \times [n]}]/ L(\gJ)^A =  \kr[\Delta(\gJ)] \]
Hence both modules identify as ideals in the respective rings (with
embeddings respecting
the multigrading).
\end{corollary}

\begin{proof}
The first part is simply the above proposition.
Since Alexander duality is an exact functor on modules,
applying this functor to
the first map,
the canonical module $\can$ is the image of the second map above.
\end{proof}

\begin{remark}
By applying Theorem \ref{SRcanonical} we immediately get that $\Delta(\gJ)$ is either a homology ball or a homology sphere. 

We will find out by combinatorial means that $\Delta(\gJ)$ is actually even a 
piecewise linear (PL) ball or sphere, see Section \ref{SphereSection} below.
\end{remark}

\subsection{The resolution}

By Proposition \ref{prop:SqLinRes},
$\gL(\can)$ is the resolution of the co-letterplace ideal $L(\gJ)$. Since
$L(\gJ|\St)$ is the Alexander dual of $\can$,
the linear complex 
$\gL(L(\gJ|\St))$ is the dual of the resolution of the co-letterplace ideal 
$L(\gJ)$. This complex has terms
\[ \gL^i(L(\gJ|\St)) = \bigoplus_{\overset {m_A \, \text{nonzero in} \, 
L(\gJ|\St)}
  {i = |A|}} S(-A^c). \]
Denote by $(m_A)^\circ$ the generator for $S(-A^c)$. The differential 
in $\gL(L(\gJ|\St))$ is given by 
\[ (m_A)^\circ \mapsto \sum_{a \in A^c} (-1)^{\alpha(a,A)} (m_{A \cup \{a\}})^\circ 
\te x_a. \]

The dual of this complex, $\gL(\can)$, is the resolution of the 
co-letterplace ideal $L(\gJ)$. 
The terms here are
\begin{equation} \label{ColpTheTerms}
 \gL_i(\can) =
 \bigoplus_{\overset{ m_A \, \text{nonzero in} \, L(\gJ|\St)}
{i = |A|}} S(-A).
\end{equation}
(So $\gL_{|P|}(\can)$ is the start of the resolution, mapping to the
minimal generators of the letterplace ideal $L(\gJ)$.)
Let $m_A$ be nonzero in $L(\gJ|\St)$ where $A \sus P \times [n]$. 
Then, since $m_A$ is divisible by some
$m_{\Gamma \phi}$, the projection map $A \mto{p_1} P$ is surjective. Let
$A_2 \sus A$ be the subset which is the union of the fibers 
of $p \in P$ with cardinality at least two. Write $e_A$ for the generator of 
$S(-A)$. We then get the following very simple explicit description 
of the resolution of $L(\gJ)$.

\begin{theorem} \label{thm:resCoLPA} The resolution of the co-letterplace ideal 
$L(\gJ)$ is $\gL(\can)$ with terms
\begin{equation*}
 \gL_i(\can) =
 \bigoplus_{\overset{ m_A \, \text{nonzero in} \, L(\gJ|\St)}
{i = |A|}} S(-A),
\end{equation*}
and, letting $e_A$ be the generator of $S(-A)$, differential 
\[ e_A \mapsto \sum_{a \in A_2} (-1)^{\alpha(a,A)} e_{A\backslash {a}} \te x_a.
\]
\end{theorem}

\begin{remark}
The criterion for $m_A$ to be nonzero in $L(\gJ|\St)$ is
straightforward to check: 
i) For each $p \in P$ let $\phi(p) = \max \{ j \, | \, 
(p,j) \in A \}$. Then  $\phi$ must be an isotone map in $\gJ$.
ii) $A$ must not contain any pairs $(p,i)$ and $(q,j)$ with $p < q$ and
$i > j$. 
\end{remark}

\begin{proof}[Proof of Theorem \ref{thm:resCoLPA}]

Let $F$ and $G$ be finitely generated free modules over a (commutative)
ring $R$, and let $F^* = \Hom(F,R)$ and similarly define $G^*$. 
Two morphisms 
\[ F \mto{\nu} G, \quad G^* \mto{\mu} F^* \]
are dual to each other provided for each $f \in F$ and $g \in G^*$
the natural pairings
\[ (\nu(f), g) = (f, \mu(g)) \]
are equal. Note that it is enough to prove this for generators $f$ and $g$
of the modules.

Consider the dual differentials
\[ \gL^{i-1}(L(\gJ|\St)) \mto{\mu} \gL^i(L(\gJ|\St)), \quad
\gL_{i}(\can) \mto{\nu} \gL_{i-1}(\can).\]
Let $e_{A^\prime}$ be the generator for the term $S(-A^\prime)$ in 
$\gL_{i}(\can)$ and $(m_A)^\circ$ the generator for the term
$S(-A^c)$ in $\gL^{i-1}(L(\gJ|\St))$ (so $m_A$ and $m_{A^\prime}$ are
nonzero in $L(\gJ|\St)$). The map $\mu$ sends
\begin{align*}
(m_A)^\circ \mapsto (-1)^{\alpha(a,A)} x_a (m_{A \cup \{a\}})^\circ
+ & \text{ terms indexed by } A \cup \{ b \} & \\
& \text{ where } b \in A^c \backslash
\{ a \}. & 
\end{align*}
Let $\nu^\prime$ be the map sending
\[ e_{A^\prime} \mapsto \sum_{a \in (A^\prime)_2} (-1)^{\alpha(a,A^\prime)} e_{A^\prime\backslash {a}} 
\te x_a.\]
We are claiming that $\nu^\prime$ is the dual map $\nu$. 

i) Suppose $A^\prime = A \cup \{ a \}$ for some $a$ (note that $a$
is then in $(A^\prime)_2$). The map $\nu^\prime$ then sends
\begin{align*} e_{A \cup \{ a \}} \mapsto (-1)^{\alpha(a,A \cup \{ a \})}
x_a e_A + & \text{ terms indexed by } (A \backslash \{b\}) \cup \{ a \} & \\ 
& \text{ where } b \in A.&  
\end{align*}
So  
\begin{align*} 
((m_A)^\circ, \nu^\prime(e_{A\cup \{ a \}})) & = (-1)^{\alpha(a,A \cup \{ a \})}
x_a & \\
(\mu(( m_A)^\circ), e_{A\cup \{ a \}}) & =  (-1)^{\alpha(a,A)}
x_a.&
\end{align*}
Since $\alpha(a,A) = \alpha(a, A \cup \{ a \})$ for $a \in A^c$ these are
equal. 

ii) If $A^\prime$ does not contain $A$ then the above pairings
are easily seen to be both zero. Hence the map $\nu^\prime$ is the
dual map $\nu$.
\end{proof}

\begin{remark} As mentioned in the introduction, another approach to get
explicitly a linear resolution of a monomial ideal $I$ in a
polynomial ring $S$
is using a regular decomposition function
\cite{HeTa}. In \cite{DoFa} and \cite{Good} one doubles it up
by even constructing cellular resolutions in 
this case. The situation of 
a monomial ideal having a regular
decomposition function is likely to be more general than what can be achieved as
regular quotients by variable differences of letterplace ideals. 
In general, however, the Alexander dual of the canonical module of the Alexander
dual $S/I^A$ would not be embedded as a multigraded ideal in a ring.

N.~Horwitz in \cite{Hor} gives explicit linear resolutions
of edge ideals under certain conditions. It is not clear to us
if his construction can be related to our setting.
\end{remark}

\subsection{Resolutions of strongly stable ideals}
\label{subsec:ResCoLpStSt}

The $i$'th homological term in $\gL(\can)$ is
\begin{equation} \label{eq:resCoLPA}
 \gL_i(\can) =
 \bigoplus_{\overset{ m_A \, \text{nonzero in} \, L(\gJ|\St)}
{i = |A|}} S(-A).
\end{equation}
By Corollary \ref{StairFacetPosetIdeal} 
the squarefree monomials $m_A$ in $L(\gJ|\St)$ are indexed by
subsets $A$ of $\cup_{\phi \in \gJ} [\Gamma \phi, T_* \phi]$. 

Hence we may write
\begin{equation*}
 \gL_{i + |P|}(\can) =
 \bigoplus_{\overset{ \overset{\phi \in \gJ}{ A \in [\Gamma \phi, T_* \phi]}}
{|A| = i + |P|}} Se_A.
\end{equation*}
Such an $A$ may be represented by a pair $(\Gamma \phi, D)$ where
\[ A = \Gamma \phi \cup D, \quad D \sus T_* \phi \backslash \Gamma \phi, 
\quad |D| = i, \]
so 
\begin{equation} \label{eq:resCoLPphiD}
 \gL_{i + |P|}(\can) =
 \bigoplus_{\overset{ \overset{\phi \in \gJ}{ D \sus T_* \phi \backslash
\Gamma \phi}}
{|D| = i}} Se_{(\Gamma \phi, D)}.
\end{equation}

The differential sends 
\[ e_A \mapsto \sum_{a \in A_2} (-1)^{\alpha(a,A)} e_{A \backslash \{ a\}}
\te x_a. \]
Note however that if $A \in [\Gamma \phi, T_* \phi]$
and $a  = (p, \phi(p))$ then $A \backslash \{a \}$
will be in a disjoint interval $[\Gamma \psi, T_* \psi]$. 
Thus using the notation (\ref{eq:resCoLPphiD}) when describing
the differential, though
more explicit, is somewhat more awkward than using the notation
(\ref{eq:resCoLPA}). However we shall now see the typical case where
this form (\ref{eq:resCoLPphiD}) is used.

Let $P = [d]$. There is a one to one correspondence between $\Hom([d],[n])$
and monomials in $\kr[x_1, \ldots, x_n]$ of degree $d$ given by
\[ \phi \mapsto m_{\phi} = \prod_{i = 1}^d x_{\phi(i)}. \]
Via this correspondence poset ideals $\gJ$ in $\Hom([d], [n])$ correspond
one to one to strongly stable ideals $I$ in $\ovS = \kr[x_1, \ldots, x_n]$ 
generated in degree $d$. As explained in \cite[Section 5]{FGH}, the ideal $I$
is a regular
quotient of the co-letterplace ideal $L(\gJ)$ in $\kr[x_{[d] \times [n]}]$
by sending $x_{i,j}$ to $x_j$. Thus the strongly stable ideal $I$ will have
a resolution $\ovgL(\can)$ with terms
\[ \ovgL_i(\can) = \oplus \ovS.e_{(\Gamma \phi, D)}.\]


\begin{lemma} Let $P = [d]$ and  
consider $(\Gamma \phi, D)$ with $D \sus T_* \phi \backslash \Gamma \phi$.

1. $D$ is uniquely determined by its projection $p_2(D)$ onto $[n]$. 

2. Any subset of $[1,\phi(d)-1]$  occurs as $p_2(D)$ for some $D$.

\end{lemma}

\begin{proof}
For $(i,j)$ in $D$ we have $\phi(i-1) \leq j < \phi(i)$ (we
define $\phi(0)$ to be $1$). When we know $\phi$, there is for each
$j < \phi(d)$ a unique position $i$ which fulfils this.
\end{proof}

  Hence the pair $(\Gamma \phi, D)$ with $|D| = i$ may be represented
by a pair $(m_\phi; j_1,\ldots, j_i)$ where $1 \leq j_1 < \cdots < j_i <
\phi(d)$. This is precisely the form of the multigraded Betti numbers
in the Eliahou-Kervaire (EK) resolution, \cite{ElKe} or \cite{PeSt}.

Although the $i$'th homological terms 
$\oplus \overline{S}.(m_\phi;j_1, \ldots, j_i)$ are the same in the EK resolution and
the resolution $\ovgL(\can)$ we derive from Theorem \ref{thm:resCoLPA} in the
case $P = [d]$, the differentials are not the same. Let us describe them in 
detail.

Fix a generator $(m_\phi; j_1, \ldots, j_i)$
in the $i$'th homological
term of the resolution. In order to define the differentials
we need to define modified isotone maps $\phi^1_t$ and $\phi^2_t$. 
For each $1 \leq t \leq i$ let $p_t \in [d]$ be such that
$\phi(p_t-1) \leq j_t < \phi(p_t)$ (in other words $(p_t,j_t) \in D$ is
the inverse image of $j_t$ by the projection map $p_2$)
and define isotone maps by
\[ \phi^1_t(r) = \begin{cases} \phi(r) & r < p_t \\
                             j_t & r = p_t \\
                            \phi(r) & r > p_t
                 \end{cases},
\quad
\phi^2_t(r) = \begin{cases} \phi(r) & r < p_t \\
                             j_t & r = p_t \\
                            \phi(r-1) & r > p_t
                 \end{cases}
\]
Let $\vardel$ map the generator $(m_\phi; j_1, \ldots, j_i)$ to
\begin{equation} \label{vardel_EK} \sum_{t = 1}^i (-1)^t x_{j_t} (m_\phi; \ldots, \hat{j_t}, \ldots), \end{equation}
and let $\mu$ map the generator to 
\begin{equation} \label{mu_EK} \sum_{t = 1}^i (-1)^{t} x_{\phi(d)} (m_{\phi^2_t}; \ldots, \hat{j_t}, \ldots ). \end{equation}
(Note that \eqref{mu_EK} can contain ill-defined summands: these are considered as zeros.)
Then the differential in the EK resolution is $\vardel - \mu$.

On the other hand, let $\vardel^{\prime}$ map $(m_\phi; j_1, \ldots, j_i)$
to 
\begin{equation} \label{vardel_coLP} \sum_{t = 1}^i (-1)^{p_t+t} x_{j_t} (m_\phi; \ldots, \hat{j_t}, \ldots), \end{equation}
and let $\mu^\prime$ map this generator to
\begin{equation} \label{mu_coLP} \sum_{(p_t, j_t) \in D'} (-1)^{p_t+t-1} 
x_{\phi(p_t)} (m_{\phi^1_t}; \ldots, \hat{j_t}, \ldots ), \end{equation}
where $D'$ is the subset of $D$ consisting of the pairs 
$(p_t,j_t)$ such that $j_t < \phi(p_t) \leq j_{t+1}$, i.e. $j_t$ is maximal
among the $j_*$'s which are strictly less than $\phi(p_t)$.

Then the differential in the resolution derived from the approach
using the co-letterplace ideal, Theorem \ref{thm:resCoLPA}, 
is $\vardel^\prime - \mu^\prime$. 
This is also (up to a sign) the differential in the cellular resolutions
of strongly stable ideals generated in a single degree, given
in \cite{Sin} and \cite{NaRe}. 

Since this differential naturally occurs both in our situation and when constructing a cellular resolution by a polytopal
complex, it is suggested that this differential is quite natural.



\section{Simplicial spheres generalizing Bier spheres} \label{SphereSection}

In this section we show that $\Delta(\gJ;n,P)$ is a simplicial ball, 
more precisely a piecewise linear (PL) ball (save for some very special
cases of $\gJ$ and $P$ when it is a sphere).
The boundary of $\Delta(\gJ;n,P)$ is then a PL sphere. This gives
a very large class of simplicial spheres associated with 
a poset $P$, a natural number $n$, and a poset ideal $\gJ$ of
$\Hom(P,[n])$. A special case of our construction here, though
still a large class, are Bier spheres, originally defined
by T.~Bier in \cite{Bier} and studied by A.~Bj\"orner et al. in \cite{BjZi}.
Bier spheres occur when $P$ is an antichain and $n = 2$. A generalization
of Bier spheres is studied by S.~Murai in \cite{MuSpheres}, see the last part 
of Subsection \ref{Subsec:BierRes}. For the definition
of piecewise linear simplicial complexes, we refer to 
\cite{Bj}.

\subsection{The simplicial ball $\Delta(\gJ)$ and its boundary}

\begin{theorem} \label{thm:BierBall}
Let $\gJ$ be a non-empty poset ideal in $\Hom(P,[n])$.
The simplicial complex $\Delta(\gJ)$ defined by the Alexander dual
$L(\gJ)^A$ is a PL-ball of codimension $|P|$ in the simplex on 
the vertex set $P \times [n]$, 
save for the following exception:
$P$ is an antichain and $\gJ = \Hom(P,[n])$. In this case
$\Delta(\gJ)$ is a PL-sphere and $L(\gJ)^A = L(n,P)$ is a complete 
intersection of
$|P|$ monomials of degree $n$.
\end{theorem}

\begin{proof}
By \cite[Theorem 5.1]{FGH} the co-letterplace ideal $L(\gJ)$ has linear
quotients and so its Alexander dual $L(\gJ)^A$ defines a shellable
simplicial complex $\Delta(\gJ)$. 
All generators of $L(\gJ)$ have cardinality $|P|$
and so
$\Delta(\gJ)$ is pure of this codimension
(by Alexander duality, the facets of $\Delta(\gJ)$ are the complements of the minimal generators of 
$L(\gJ)$). We now use the criterion
\cite[Theorem 11.4]{Bj}: a (pure) shellable
simplicial complex such that every codimension one face is in
at most two facets is a PL ball or a PL sphere.

Let $\phi \in \gJ$ be an isotone map. The complement of the graph 
$\Gamma \phi$ is a facet of $\Delta(\gJ)$. 
Let $G$ be a codimension one
face of this facet. Then the complement $G^c$ contains $\Gamma \phi$
and the cardinality of $G^c$ is one more than that of $\Gamma \phi$.
Considering the projection of $G^c$ to $P$, there will be exactly one
fiber of cardinality two, the others of cardinality one. Thus $G^c$
can contain at most two graphs of isotone maps, and $G$ is in at most
two facets. Hence $\Delta(\gJ)$ is a ball or a sphere.

\medskip
We now show that $\Delta(\gJ)$ is a ball when
i) $P$ is not an antichain or ii) $\gJ$ is nonempty and properly included in 
$\Hom(P,[n])$. We show that there is $G$ as above contained in only one
facet.

i) Let $q < q'$ in $P$, and let $\phi$ be the map that sends everything to $1$ 
($\phi$ lies in any nonempty poset ideal). 
Then the set $\{ (p,i) \, | \, p \in P, i \geq 2\}$ is a 
facet $F$ of $\Delta(\gJ)$.
Now consider the face $G$ obtained from $F$ by removing the element $(q, 2)$. 
The link of $G$ consists of $(q, 2)$ only (it cannot contain $(q, 1)$, 
since this would give a non-isotone map sending $q$ to $2$ and $q'$ to $1$).
Hence the link of $G$ is just a point.

ii) If the poset ideal is properly included in $\Hom(P, [n])$, 
there exist two isotone 
maps $\phi$ and $\psi$ that differ by a single entry ($\phi$ sends $p$ to $i$, 
$\psi$ sends 
$p$ to $j$ with $i < j$) and are such that $\phi$ lies in the poset ideal 
but $\psi$ does not. We can consider the facet $F$ which is the complement 
of the graph 
of $\phi$ and then take the face $G$ consisting of $F$ without the vertex 
$(p, j)$. 
The link of $G$ consists of $(p,j)$ only by construction.

The rest of the claim follows from \cite[Corollary 2.4]{FGH}.
\end{proof}

As we saw in Proposition \ref{pro:ResCoLpCanD},
the Alexander dual $(\can)^*$ of the canonical module of 
$\kr[\Delta(\gJ)] = \kr[x_{P\times[n]}]/L(\gJ)^A$
is the image of the natural map
\begin{equation} \label{eq:GorstOmegaD} L(\gJ) \pil \kr[x_{P\times[n]}]/B(\gJ).
\end{equation}
and, by Corollary \ref{cor:ResCoLpBoth},
the canonical module $\can$ is the image of the natural map
\begin{equation} \label{eq:GorstOmega}
B(\gJ)^A \pil \kr[x_{[n] \times P}]/L(\gJ)^A 
\end{equation}
and so identifies as an ideal of its ring.
This is the situation described in Theorem \ref{SRcanonical}.



Let us describe the boundary of our ball. The map (\ref{eq:GorstOmegaD}) 
gives an exact sequence
\begin{align*} 
 0  \pil  L(\gJ) \cap B(\gJ) & \pil L(\gJ) \pil \kr[x_{P\times[n]}]/B(\gJ) & \\
& \pil  \kr[x_{P\times[n]}]/(L(P,n) + B(P,n)) \pil 0.& 
\end{align*}
Alexander duality is an exact functor and, hence, dualizing the above
sequence, we get an exact sequence
\begin{align} \label{eq:GorstSigmaSes}
 0  \pil L(P,n)^A \cap B(P,n)^A & \pil B(\gJ)^A \mto{\tau}
\kr[\Delta(\gJ)] & \notag \\
&\pil \kr[x_{[n] \times P}]/(L(\gJ)^A + B(\gJ)^A) \pil 0. &
\end{align}
Here we use the fact that the Alexander dual ideal of a sum of ideals is the intersection
of their Alexander dual ideals, and vice versa. 
We get a squarefree quotient ring
of $\kr[\Delta(\gJ)]$ and so a natural subcomplex of $\Delta(\gJ)$. 
The image of $\tau$ is the canonical module $\omega_{\kr[\Delta]}$.
We then see from Sequence (\ref{eq:LinresCanBound}) that this subcomplex
is the boundary $\Sigma = \partial \Delta$, with Stanley-Reisner ideal
$L(\gJ)^A + B(\gJ)^A$. Note that the above Sequence (\ref{eq:GorstSigmaSes}) 
gives a presentation of the canonical module
\[ 0 \pil L(P,n)^A \cap B(P,n)^A \pil B(\gJ)^A \pil \omega_{\kr[\Delta]} 
\pil 0. \]
This is different from the presentation used in 
(\ref{boundarysequence}):
\[ 0 \pil L(\gJ)^A \pil L(\gJ)^A + B(\gJ)^A \pil \omega_{\kr[\Delta]} 
\pil 0. \]
We summarize in the following:

\begin{theorem} \label{explicitboundary} Let $\gJ$ be a non-empty poset ideal in $\Hom(P,[n])$
and exclude the following case: $P$ is an antichain and $\gJ = \Hom(P,[n])$.
Then the ideal in $\kr[x_{P \times [n]}]$
\[ (L(\gJ) \cap B(\gJ))^A = L(\gJ)^A + B(\gJ)^A = L(\gJ)^A + L(\gJ^c)^A \cap
B(P,n)^A \]
is a Gorenstein ideal of codimension $|P| +1$ which defines the 
boundary $\Sigma(\gJ)$ of the ball $\Delta(\gJ)$. As a consequence, $\Sigma(\gJ)$ is a sphere.
\end{theorem}


\begin{corollary}
When $P$ is an antichain and $n = 2$, the $\Sigma$ are
the Bier spheres studied by A.~Bj\"orner et al. in \cite{BjZi},
originally defined by T.~Bier \cite{Bier}.
\end{corollary}

\begin{proof} Let $P$ be the antichain on $d$ elements.
A poset ideal $\gJ$ in $\Hom(P,[2])$ corresponds to a 
simplicial complex $X$ on $d$ elements. Let $X^A$ be the Alexander
dual simplicial complex. We consider the Stanley-Reisner ideal $I_X$
of $X$ to live in $\kr[x_{P \times \{ 1\}}]$ and the Stanley-Reisner
ideal $I_{X^A}$ of $X^A$ to live in $\kr[x_{P \times \{ 2 \}}]$. 
By the start of \cite[Section 3]{HeKa}, the Stanley-Reisner
ideal of the Bier sphere associated with $X$ has ideal in $\kr[x_{P \times [2]}]$
generated by
\[ I_X + I_{X^A} + (x_{p,1}x_{p,2})_{p \in P}. \]
But by the last paragraph of Subsection 6.4 in \cite{FGH} this
is the ideal $L(\gJ)^A + L(\gJ^c)^A$ which coincides with
the ideal in Theorem \ref{explicitboundary} above, since $B(P,n)$ is the zero ideal when
$P$ is an antichain.
\end{proof}

\subsection{Restricted versions}
\label{Subsec:BierRes}

An isotone map $\mu\colon P \pil [n]$ such that for every $\phi$
in the poset ideal $\gJ \sus \Hom(P,[n])$ we have $\phi \leq \mu$
is said to be an {\it upper bound} for $\gJ$. Note that the minimal generators
of $L(\gJ)$ involve no variables $x_{p,i}$ where $i > \mu(p)$.
 Let $(P\times [n];\mu) = \{(p,i) \in P \times [n] \,| \, i \leq \mu(p) \}$ and let $\kr[x_{P \times [n];\mu}]$ be the polynomial ring whose variables
are indexed by this set.

Then all generators of $L(\gJ)$ are in this polynomial ring, and
we denote by $L^{\leq \mu}(\gJ)$ the ideal in $\kr[x_{P \times [n];\mu}]$ 
which they generate. The generators of the Alexander dual $L(\gJ)^A$
will similarly only involve variables indexed by $(P\times [n];\mu)$,
and so these generators generate an ideal $L^{\leq \mu}(\gJ)^A$ in
$\kr[x_{P \times [n];\mu}]$. (Note that, in this situation, taking the Alexander dual
and restricting the ideal to the smaller polynomial ring commute. 
This is not true in general
when the generators of the ideal are not in the smaller polynomial ring.) 
We denote the associated simplicial complex
by $\Delta^{\leq \mu}(\gJ)$. We see that $\Delta(\gJ)$ is the cone
over $\Delta^{\leq \mu}(\gJ)$ over the vertex set 
$P \times [n] \backslash (P\times [n];\mu)$.
Note that there is a unique minimal $\mu$ such that one can do this:
since $\Hom(P,[n])$ is a lattice (in fact a distributive lattice),
one can let $\mu$ be the join of the elements in $\gJ$.
This minimal $\mu$ is explicitly given by $\mu(p) = \max \{ \phi(p) \, |\, \phi
\in \gJ \}$ and could be called the {\it hull} of $\gJ$. 

All statements in this paper continue to hold in the more general
setting with an upper bound $\mu$, with the appropriate modifications:
each ideal $I$ in $\kr[x_{P \times [n]}]$ must be replaced with the 
restricted ideal
$I^{\leq \mu} = I \cap \kr[x_{P \times [n];\mu}]$. For instance 
$\Delta^{\leq \mu}(\gJ)$ is still a ball, and the canonical module
of $\kr[x_{P \times [n];\mu}]/L^{\leq \mu}(\gJ)^A$ will be the image
of 
\[ (B^{\leq \mu}(\gJ))^A  \pil \kr[x_{P \times [n];\mu}]/L^{\leq \mu}(\gJ)^A
= \kr[\Delta^{\leq \mu}(\gJ)]. \]

For homological purposes considering an ideal in a polynomial ring
or the ideal generated by the ideal in a larger polynomial ring amounts to
the same. We get the minimal resolution from the latter by tensoring
up the former with a polynomial ring. 

However when it comes to boundaries things are different.
The boundary $\Sigma^{\leq \mu}(\gJ)$ will have different homological
behaviour than $\Sigma(\gJ)$, although the relationship
between them is simple.

\medskip
In \cite{MuSpheres} S.~Murai 
generalizes Bier spheres.
In our setting this relates as follows. 
Let $P$ be the antichain $\underline{d} = \{1, 2, \ldots, d\}$ on $d$ elements.
Take a sequence of natural numbers $c_1, \ldots, c_d$
and a larger natural number $n$.
Murai then considers poset ideals $\gJ$ of $\Hom(\underline{d},[n])$,
with the restriction that $\phi(i) \leq c_i$ for every $\phi \in \gJ$ and $i$.
Thus the map $\mu$ given by $\mu(i) = c_i$ is an upper bound for $\gJ$.
Murai shows \cite[Lemma 1.4]{MuSpheres} that $\Delta^{\leq \mu}(\gJ)$ is a ball 
unless $\gJ$ is the poset ideal of all maps $\phi$ with
$\phi(i) \leq c_i$ for all $i$. He calls 
this a Bier ball. His generalized Bier spheres are the boundaries
of these balls. (When $\gJ$ is the poset ideal of all elements
$\leq \mu$, then $\Delta^{\leq \mu}(\gJ)$ is a sphere, loc.~cit.)

The poset ideal $\gJ$ corresponds to a finite poset ideal in
$[n]^d \sus \NN^d$, 
also called a multicomplex. Such a finite multicomplex corresponds to an 
artinian monomial
ideal $J \sus \kr[x_1, \ldots, x_d]$, 
which is $\mu$-determined in the language of \cite{Mil}.
The Alexander dual $L^{\leq \mu}(\gJ)^A$, which is the 
Stanley-Reisner ideal corresponding to $\Delta^{\leq \mu}(\gJ)$, is then the 
(standard) polarization of the artinian monomial ideal $J$ considered in 
the polynomial ring
$\kr[x_{P \times [n];\mu}]$, see \cite[Lemma 3.2]{MuSpheres}.

\begin{remark}
In \cite{BjZi} they also construct a {\it Bier poset} $B(P,I)$ for any 
poset ideal $I$ in $P$. This is however a construction not seemingly
related to ours. The order complexes of these Bier posets are also not in 
general homology spheres.
\end{remark}

\bibliographystyle{amsplain}
\bibliography{Bibliography}

\end{document}